\documentclass[a4paper,11pt]{amsart}

\usepackage[english]{my-shortcuts}
\usepackage{srcltx,algorithm,algorithmic}
\usepackage{a4wide}

\usepackage{enumitem}

\usepackage{tikz}
\begin{document}

\title[Sensing tensors with Gaussian filters]
{Sensing tensors with Gaussian filters}
\author{St\'ephane Chr\'etien {\tiny and} Tianwen Wei}

\address{Laboratoire de Math\'ematiques, UMR 6623\\ 
Universit\'e de Franche-Comt\'e, 16 route de Gray,\\
25030 Besancon, France} 
\email{stephane.chretien,tianwen.wei@univ-fcomte.fr}

\maketitle

\begin{abstract}
Sparse recovery from linear Gaussian measurements has been the subject of much investigation since the 
breaktrough papers \cite{CRT:IEEEIT06} and \cite{donoho2006compressed} on Compressed Sensing. Application to 
sparse vectors and sparse matrices via least squares penalized with sparsity promoting norms 
is now well understood using tools such as Gaussian mean width, statistical
dimension and the notion of descent cones \cite{tropp2014convex} \cite{Vershynin:ArXivEstimation14}. Extention of these ideas to 
low rank tensor recovery is starting to enjoy considerable interest due to its many potential applications to 
Independent Component Analysis, Hidden Markov Models and Gaussian Mixture Models \cite{AnandkumarEtAl:JMLR14}, 
hyperspectral image analysis \cite{zhang2008tensor}, to name a few. In this paper, we demonstrate that the recent approach of 
\cite{Vershynin:ArXivEstimation14} provides very useful error bounds in the tensor setting using the nuclear norm or the 
Romera-Paredes--Pontil \cite{RomeraParedesPontil:NIPS13} penalization.    
\end{abstract}

\section{Introduction}
Real tensors, i.e. multidimensional arrays of real numbers, 
have been recently a subject of great interest in the applied mathematics community.  
We refer to  \cite{Landsberg:Tensors12} and \cite{Hackbusch:TensorSpaces12} for modern references on this subject. 
It has become quite clear nowadays that real symmetric tensors such as cumulants up to fourth order play a very 
important role in many applications in statistics, machine learning and signal processing; see for instance 
\cite{KoldaBader:SIREV09} for a general survey. Research on applications of tensors 
has been increasing in the recent years with very important conceptual contributions such as proposed in \cite{AnandkumarEtAl:NIPS12},
\cite{AnandkumarEtAl2:NIPS12}, \cite{ComonMourrain:SIMAX08}, \cite{BrachatEtAl:LAA10} and the very nice survey of applications in 
\cite{AnandkumarEtAl:JMLR14}. In particular, certain Gaussian Mixture Models (GMM) can be estimated using this approach. The same is also true for 
Independent Component Analysis (ICA) and Hidden Markov Models (HMM). Nonsymmetric tensors also occur frequently in applications as 3D images 
such as in medical imaging and hyperspectral image processing \cite{KoldaBader:SIREV09}, \cite{ZhangEtAl:JOSA08}.

In some applications, the tensor is observed through the operation of random filtering, e.g. taking the scalar product with an i.i.d. Gaussian
random vector. 
Gaussian random sensing has been thoroughly investigated in recent years and can be recast as sparse recovery 
for one dimensional tensors (i.e. vectors), low rank recovery for bidimensional tensors (i.e. matrices). See e.g. \cite{Vershynin:ArXivEstimation14} 
for a tutorial on this topic.

The goal of this short note is to show that the results of \cite{Vershynin:ArXivEstimation14} can easily be extended to tensors. First, we consider 
nuclear norm minimization. Next, we consider recovery via minimization of the Romero-Paredes--Pontil functional. 

\section{Main facts about tensors}
Let $D$ and $n_1,\ldots,n_D$ be positive integers. Let $\Xc \in \mathbb R^{n_1\times \cdots \times n_D}$ denote a 
$D$-dimensional array of real numbers. We will also denote such arrays as tensors. 
\subsection{Basic notations and operations}
A subtensor  of $\Xc$ is a tensor obtained by fixing some of its coordinates. As an example, fixing
one coordinate $i_d=k$ in $\Xc$ for some $k\in\{1,\ldots, n_d\}$ yields a tensor in
$\Rb^{n_1\times\cdots\times n_{d-1} \times n_{d+1}\times \cdots \times n_D}$.
 In the sequel, we will denote this subtensor of  $\Xc$ by $\Xc_{i_d=k}$.
 
The fibers of a tensor are particular subtensors that have only one mode, i.e. obtained by fixing every coordinate except one. The mode-$d$ fibers are the vectors
\bean 
\left(\Xc_{i_1,\ldots,i_{d-1},i_d,i_{d+1},\ldots,i_D}\right)_{i_d=1,\ldots,n_d}.
\eean 
They extend the notion of columns and rows from the matrix to the tensor framework. 
For a matrix, the mode-1 fibers are the columns and the mode-2 fibers are the rows. 

The mode-$d$
matricization $\Xc_{(d)}$ of $\Xc$ is obtained by forming the matrix whose rows are the mode-$d$ fibers 
of the tensor, arranged in an cyclic ordering; see \cite{KoldaBader:SIREV09} for details. 

The mode-$d$ multiplication of a tensor $\Xc\in\Rb^{n_1\times\cdots \times n_D}$ by a matrix $U \in \Rb^{n_d'\times n_d}$, 
denoted by $\Xc\times_d U$, gives a tensor in $\Rb^{n_1\times\cdots \times n_d'\times\cdots \times n_D}$. It is  defined as 
\bean 
(\Xc\times_d U)_{i_1,\ldots,i_{d-1},i_d^\prime,i_{d+1},\ldots,i_D} 
& = & \sum_{i_d=1}^{n_d} \Xc_{i_1,\ldots,i_{d-1},i_d,i_{d+1},\ldots,i_D} U_{i_d',i_d}.
\eean 

Last, we denote by $\|\cdot\|_{F}$ the Frobenius norm, i.e.:
\begin{align*}
\|\Xc\|_{F} & = \left(\sum_{i_1=1}^{n_1}\cdots\sum_{i_D=1}^{n_D}\Xc_{i_1,\ldots,i_D}^2  \right)^{1/2},\quad 
\forall\Xc\in\Rb^{n_1\times \ldots\times n_D}.
\end{align*}


\subsection{Higher Order Singular Value Decomposition (HOSVD)}
The Tucker decomposition of a tensor is a very useful decomposition. It can be chosen so that 
after appropriate 
orthogonal transformations, one can reveal a tensor $\Sc$ hidden inside $\Xc$ 
enjoying interesting rank and orthogonality properties.
In this contribution, we will make use of the HOSVD, a generalization of the matrix SVD
to the tensor setting based on Tucker decomposition. 
\begin{theo}[HOSVD \cite{DeLathauwerEtAl:SIAMAX00}]
Every tensor $\Xc\in\Rb^{n_1\times \cdots \times n_D}$ can be written as
\bea\label{Tucker}
\Xc & = & \Sc(\Xc) \times_1 U^{(1)} \times_2 U^{(2)} \cdots \times_D U^{(D)}, 
\eea
where each $U^{(d)}\in\Rb^{n_d\times n_d}$ is an orthogonal matrix and 
$\Sc(\Xc)\in\Rb^{n_1\times \cdots \times n_D}$ is a tensor of the same size as $\Xc$ with the following properties:
\begin{itemize}
\item[1.] For all possible values of $d,\alpha$ and $\beta$ subject to 
$\alpha\neq \beta$, subtensors $\Sc(\Xc)_{i_d=\alpha}$ and $\Sc(\Xc)_{i_d=\beta}$ are orthogonal, i.e.
\bean
\langle \Sc(\Xc)_{i_d=\alpha}, \Sc(\Xc)_{i_d=\beta}\rangle =0. 
\eean
\item[2.] For all possible values of $d$, there holds :
\bean
\|\Sc(\Xc)_{i_d=1}\|_{F}\geq \| \Sc(\Xc)_{i_d=2} \|_{F}\geq \cdots \geq \| \Sc(\Xc)_{i_d=n_d} \|_{F}\geq 0.
\eean
\item[3.] The quantities $\|\Sc(\Xc)_{i_d=k}\|_{F}$ for $k=1,\ldots,n_d$ are the singular values of 
the mode-$d$ matricization $\Xc_{(d)}$ of $\Xc$ and the columns of $U^{(d)}$ are the corresponding singular vectors.
\end{itemize}
\end{theo}

Let $\otimes$ denote the standard Kronecker product 
for matrices. Then it follows from (\ref{Tucker}) that
\bea
\label{matriz}  \label{1021b}
\quad\quad\Xc_{(d)}=U^{(d)}  \Sc_{(d)}(\Xc)  \left(U^{(d+1)} \otimes \cdots \otimes U^{(D)} \otimes U^{(1)} \otimes \cdots \otimes U^{(d-1)}\right)^t,
\eea 
where $\Sc(\Xc)_{(d)}$ is the mode-$d$ matricization of $\Sc(\Xc)$.
Taking the (usual) SVD of the matrix $\Xc_{(d)}$
\bean 
\Xc_{(d)} & = & U^{(d)} \Sigma^{(d)} {V^{(d)}}^t
\eean  
and based on (\ref{matriz}), we get  
\bean 
\Sc_{(d)}(\Xc) & = & \Sigma^{(d)} {V^{(d)}}^t\left(U^{(d+1)} \otimes \cdots \otimes U^{(D)} \otimes U^{(1)} 
\otimes \cdots \otimes U^{(d-1)}\right).
\eean

\subsection{The spectrum}
The mode-$d$ spectrum is defined as the vector of singular values of $\Xc_{(d)}$ and
we will denote it by $\sigma^{(d)}(\Xc)$. 
Notice that this construction implies that $\Sc(\Xc)$ has orthonormal fibers for every modes. With a slight abuse of notation, we will denote by $\sigma$ 
the mapping which to each tensor $\Xc$ assigns the vector $1/\sqrt{D}\: (\sigma^{(1)},\ldots,\sigma^{(D)})$ of all mode-$d$ singular spectra.

\subsection{Tensor norms}
We can define several tensor norms on $\Rb^{n_1\times \cdots \times n_D}$. The first one is a natural extension of the 
Frobenius norm or Hilbert-Schmidt norm from matrices to tensors. We start by defining the following scalar product 
on $\Rb^{n_1\times \cdots \times n_D}$: 
\bean
\langle \Xc, \Yc\rangle & \defeq & \sum_{i_1=1}^{n_1}\cdots  \sum_{i_D=1}^{n_D} x_{i_1,\ldots,i_D} y_{i_1,\ldots,i_D}.   
\eean 
Using this scalar product, we can also define the Frobenius norm as
\bean 
\|\Xc\|_F & \defeq & \sqrt{\langle \Xc,\Xc\rangle}.
\eean
One may also define an "operator norm" in the same manner as for matrices as follows
\bean 
\|\Xc\|_{\phantom{S}} &\defeq & \max_{\stackrel{\displaystyle u^{(d)} \in\Rb^{n_d}, \ \|u^{(d)}\|_2=1}{d=1,\ldots,D}}
\langle\Xc, u^{(1)}\otimes\cdots \otimes u^{(D)}\rangle
\eean 
We also define the 
\bea \label{nuclearnorm}
\|\Xc\|_* & \defeq & \max_{\stackrel{\Yc\in\Rb^{n_1\times \cdots \times n_D}}{ \|\Yc\|\leq 1}}\langle \Xc, \Yc\rangle.
\eea
The norm $\Vert \cdot \Vert_*$ can be interpreted as a generalization of the nuclear norm for matrices. They can be shown to be 
equal for certain class of tensors such as Orthogonally Decomposable tensors \cite{ChretienWei:Prep15}. 
Another interesting function is the Romera-Paredes--Pontil functional.

\section{Tensor recovery based on random measurements using convex optimization}
\subsection{Previous works}
Our goal is to estimate an unknown but structured $n_1\times \cdots \times n_D$ tensor $\Xc^\sharp$ from $m$ linear observations
given by 
\bean
y_i=\langle \Gc_i, \Xc^\sharp\rangle, \quad i=1,\dots, m.
\eean
The unknown tensor of interests $\Xc^\sharp$ although resides in a extremely high dimensional data space, it has a low-rank structure in many applications. The general problem of estimating a low rank tensor has applications in many different areas, both theoretical and applied. We refer the readers to 
\cite{AnandkumarEtAl2:NIPS12, icml2013_romera-paredes13}.

For a tensor of low Tucker rank, the matrix unfolding along each mode has low rank. Given observations
$y_1,\ldots, y_m$, we would like to attempt to recover $\Xc^\sharp$ by minimizing some combination of the ranks of the unfoldings, over all tensors $\Xc^\sharp$ that are consistent with our observations. This yields the following optimization problem:
\bea\label{l0min}
\min_{\Xc \in \mathbb R^{n_1 \times \cdots \times n_D}} \sum_{d=1}^d \rank(\Xc_{(d)})\quad \textrm{ s.t. } \quad \langle \Gc_i, \Xc \rangle = y_i, \: i=1,\ldots, m.
\eea
Optimization problem (\ref{l0min}), although intuitive, is non-convex and NP hard. A natural convex surrogate  of (\ref{l0min}) can be obtained by replacing the rank with matrix nuclear norms \cite{fazel2002matrix}.  The resulting optimization problem becomes
\bea\label{l1min}
\min_{\Xc \in \mathbb R^{n_1 \times \cdots \times n_D}} \sum_{d=1}^d \|\Xc_{(d)}\|_*\quad \textrm{ s.t. } \quad \langle \Gc_i, \Xc \rangle = y_i, \: i=1,\ldots, m.
\eea
This optimization problem was first introduced by \cite{6138863, signoretto2010nuclear} and has been used successfully in a number of applications \cite{5651225}. 

In this work, we are going to consider a related but somewhat different optimization problem:
\bea\label{NuclearOptimization}
\min_{\Xc \in \mathbb R^{n_1 \times \cdots \times n_D}} \|\Xc\|_* \quad \textrm{ s.t. } \quad \langle \Gc_i, \Xc \rangle = y_i, \: i=1,\ldots, m,
\eea
where the tensor nuclear norm $\|\cdot\|_*$ is defined in (\ref{nuclearnorm}). We point out that in general 
$\|\Xc\|_*\neq \frac{1}{D}\sum_{d=1}^D \|\Xc_{(d)}\|_*$
but the two quantities coincides when $\Xc$ is an orthogonally decomposable tensor.
The main reason to consider (\ref{NuclearOptimization}) is that the method established in \cite{Vershynin:ArXivEstimation14} for Gaussian random filter can be easily generalized to the tensor framework.

\subsection{Recovery by nuclear norm minimization}

\begin{theo}
\label{main1}
Assume that $\Gc_i$, $i,\ldots,m$ are independent $n_1\times \cdots \times n_D$ random Gaussian tensors with independent $\Nc(0,1)$ entries.
Let $\hat{\Xc}$ to be a solution of the convex program
\bean
\min_{\Xc \in \mathbb R^{n_1 \times \cdots \times n_D}} \|\Xc\|_* \quad \textrm{ s.t. } \quad \langle \Gc_i, \Xc \rangle = y_i, \: i=1,\ldots, m.
\eean
Then
\bea \label{errNuc} 
\Eb \left\|\hat{\Xc} -\Xc^\sharp\right\|_F\leq 4\sqrt{\pi} \: \frac{\sqrt{n_1} + \cdots + \sqrt{n_D}}{m} \: \|\Xc^\sharp\|_*.
\eea
\end{theo}
\begin{proof}
We consider the following set of $n_1\times \cdots \times n_D$ tensors:
\bean
K\defeq \{\Xc: \|\Xc\|_*\leq \|\Xc^\sharp\|_*\}.
\eean
Applying Theorem 6.2 from \cite{Vershynin:ArXivEstimation14}, we obtain
\bean
\Eb\left[ \sup_{\stackrel{\Xc \in K}{\langle \Gc_i, \Xc \rangle = y_i, \: i=1,\ldots, m}}\|\hat{\Xc} -\Xc\|_F \right] \leq \sqrt{2\pi} \cdot \frac{w(K)}{\sqrt{m}},
\eean
where $w(K)$ denotes the Gaussian mean width \cite{Vershynin:ArXivEstimation14} of set $K$.
By the symmetry of $K$, we have
\begin{align*}
w(K) =\Eb \left[\sup_{\Xc\in K-K}\langle \Gc, \Xc \rangle\right] & =2\: \Eb \sup_{\Xc\in K}\langle \Gc, \Xc \rangle,
\end{align*}
where $\Gc$ is a Gaussian random tensors with $\Nc(0,1)$ entries. Then using the inequality $\langle \Gc, \Xc \rangle \leq \|\Gc\|\cdot \|\Xc\|_*$ and Lemma 
\ref{lemmaA}, we obtain
\begin{align*}
w(K) & \leq 2 \: \Eb \left[\sup_{\Xc\in K} \|\Gc\|\cdot \|\Xc\|_*\right] \leq 2 \: (\sqrt{n_1} + \cdots + \sqrt{n_D}) \: \|\Xc^\sharp\|_*.
\end{align*}
Then bound (\ref{errNuc}) follows.
\end{proof}

\subsection{The Romera-Paredes--Pontil relaxation}
In the sequel, let us denote by $N =\sum_{i=1}^Dn_i$.
The Romera-Paredes--Pontil function on $\mathbb R^{N}$, denoted by $\omega_\alpha^{**}$, is the convex envelope of the 
cardinality function $\Vert \cdot\Vert_0$ on the $\ell_2$-ball of radius $\alpha$. Its conjugate $\omega_\alpha^*$ is defined as 
\begin{align}
\label{dualRPP}
\omega_\alpha^{*}(g) & = \sup_{\Vert s\Vert_2\le \alpha} \: \langle g,s\rangle-\Vert s\Vert_0.
\end{align}
By conjugate duality, we have for all $g$, $s$ in $\mathbb R^{N}$
\begin{align}
\label{conjuRPP}
\langle g,s\rangle & \le \omega_\alpha^{*}(g)+\omega_\alpha^{**}(s). 
\end{align}
Taking $g$ such that $\omega_\alpha^{*}(g)=1$ and $\omega_\alpha^{**}(s)=1$, we have 
\begin{align*}
\langle g,s\rangle & \le 2. 
\end{align*}
Therefore, for any $g$ and $s$, we have 
\begin{align}
\label{holdRPP}
\langle g,s\rangle & \le 2 \: \omega_\alpha^{*}(g) \: \omega_\alpha^{**}(s).
\end{align}
Moreover, by the Von Neumann's trace inequality for tensors \cite{ChretienWei:LAA15}, we have
\begin{align}
\label{vn}
\langle \Gc,\Xc \rangle & \le \langle \sigma(\Gc),\sigma(\Xc) \rangle 
\end{align}
combining (\ref{holdRPP}) and (\ref{vn}), we obtain that 
\begin{align}
\label{tensholdRPP}
\langle \Gc,\Xc \rangle & \le 2 \: \omega_\alpha^{*}(\sigma(\Gc)) \: \omega_\alpha^{**}(\sigma(\Xc)).   
\end{align}
based on these results, we obtain the following theorem.
\begin{theo}
Assume that $\Gc_i$, $i,\ldots,m$ are independent $n_1\times \cdots \times n_D$ random Gaussian tensors with independent $\Nc(0,1)$ entries.
Let $\hat{\Xc}$  be a solution of the convex program
\begin{align}
\min_{\Xc \in \mathbb R^{n_1 \times \cdots \times n_D}} \quad \omega_\alpha^{**}(\sigma(\Xc)) \quad \textrm{ s.t. } \quad \langle \Gc_i, \Xc \rangle = y_i, \: i=1,\ldots, m.
\end{align}
Then
\begin{align}
\label{errRPP} 
\Eb\left[\|\hat{\Xc} -\Xc^\sharp\|_F\right] & \leq 
\frac{64 \: \alpha \sqrt{\pi}}{\sqrt{m}}   
\Bigg[\Bigg(\sqrt{\log (p)} \frac{(1+N)^{D+1}-2^{D+1}}{D+1} + \frac{1}{4}\sqrt{\pi q} \Big(N+ \frac{N(1+N)^D-2^D}{D} \nonumber \\
& +\frac{2^{D+1}-(1+N)^{D+1}}{D(D+1)}\Big) \Bigg) -\frac{N(1+N)}2 \Bigg] \omega_\alpha^{**}(\sigma(\Xc^\sharp)),
\end{align}
where 
$n^*=\max\{n_1,\ldots,n_D\}$, $p =C_{D-1}^{r-1} 3^{rn_*}$ and $q=rn_*+ (r/D)^D$.
\end{theo}
\begin{proof}
We consider the following set of $n_1\times \cdots \times n_D$ tensors:
\bean
K\defeq \{\Xc: \omega_\alpha^{**}(\sigma(\Xc))\leq \omega_\alpha^{**}(\sigma(\Xc^\sharp))\}.
\eean
Exactly as for the proof of Theorem \ref{main1}, applying Theorem 6.2 from \cite{Vershynin:ArXivEstimation14}, we obtain
\bean
\Eb\left[ \sup_{\stackrel{\Xc \in K}{\langle \Gc_i, \Xc \rangle = y_i, \: i=1,\ldots, n}}\|\hat{\Xc} -\Xc\|_F \right] \leq \sqrt{2\pi} \cdot \frac{w(K)}{\sqrt{m}}.
\eean
and using the obvious symmetry of $K$, we have
\begin{align*}
w(K) =\Eb \left[\sup_{\Xc\in K-K}\langle \Gc, \Xc \rangle\right] & =2\: \Eb \left[\sup_{\Xc\in K}\langle \Gc, \Xc \rangle\right],
\end{align*}
where $\Gc$ is a Gaussian random tensors with $\Nc(0,1)$ entries. Then using inequality (\ref{tensholdRPP}), we obtain
\begin{align*}
w(K) & \leq 4 \: \mathbb E\Bigg[ \omega_\alpha^{*}(\sigma(\Gc))\Bigg] \: \omega_\alpha^{**}(\sigma(\Xc)) 
\end{align*}
and using Lemma \ref{lemmaB}, we get 
\begin{align*}
w(K) & \leq 32 \sqrt{2} \: \alpha \:  
\Bigg(\Bigg(\sqrt{\log (p)} \frac{(1+N)^{D+1}-2^{D+1}}{D+1} + \frac{1}{4}\sqrt{\pi q} \Big(N+ \frac{N(1+N)^D-2^D}{D} \nonumber \\
& +\frac{2^{D+1}-(1+N)^{D+1}}{D(D+1)}\Big) \Bigg) -\frac{N(1+N)}2 \Bigg) \omega_\alpha^{**}(\sigma(\Xc^\sharp)).
\end{align*} 
Then bound (\ref{errRPP}) follows. 
\end{proof}

\section{Some results on Gaussian tensors}

\subsection{The spectral norm of a Gaussian tensor}

\begin{lemm}\label{lemmaA}
Let 
\bean 
\Xc=\left(\Xc_{i_1,\ldots,i_D}\right)_{i_1=1,\ldots,n_1,\ldots,i_D=1,\ldots,n_D}.  
\eean 
be a tensor with i.i.d. standard Gaussian entries. Then
\begin{align*}
\Eb\left[ \left\|\Xc\right\|\right] & \leq \sum_{i=1}^D\sqrt{n_i}.
\end{align*}
\end{lemm}
\begin{proof} Consider the following stochastic process indexed by $(u^{(1)},\ldots,u^{(D)})$, where $u^{(d)}\in \Sc^{n_d-1}$ for $d=1,\ldots,D$
\bean
X_{u^{(1)},\ldots,u^{(D)}}\defeq \langle \Xc, u^{(1)}\otimes \cdots \otimes u^{(D)} \rangle.
\eean
Let us compute the variance of the increments of this random process.
For any $u^{(d)},v^{(d)}\in \Sc^{n_d-1}$ for $d=1,\ldots,D$,  we have
\begin{align*}
& \hspace{-2cm} \Eb\left[\left\|X_{u^{(1)},\ldots,u^{(D)}} - X_{v^{(1)},\ldots,v^{(D)}} \right\|_F^2\right] \\
& = \Eb\left[  \sum_{i_1,\ldots,i_D=1}^{n_1,\ldots,n_D} \Xc_{i_1,\ldots,i_D}^2 
\Big(\prod_{d=1}^D u_{i_d}^{(d)} - \prod_{d=1}^D v_{i_d}^{(d)} \Big)^2\right] 
\end{align*} 
and since the entries of $\Xc$ have unit variance, we get 
\begin{align*}
& \hspace{-2cm} \Eb\left[\left\|X_{u^{(1)},\ldots,u^{(D)}} - X_{v^{(1)},\ldots,v^{(D)}} \right\|_F^2\right] \\
& = \sum_{i_1,\ldots,i_D=1}^{n_1,\ldots,n_D} 
\Big(\prod_{d=1}^D u_{i_d}^{(d)} - \prod_{d=1}^D v_{i_d}^{(d)} \Big)^2 \\
& = \left\| u^{(1)}\otimes\cdots\otimes u^{(D)} - v^{(1)}\otimes\cdots\otimes v^{(D)}  \right\|_F^2 
\end{align*}
Moreover, since 
\begin{align*}
\left\| u^{(1)}\otimes\cdots\otimes u^{(D)} - v^{(1)}\otimes\cdots\otimes v^{(D)}  \right\|_F^2 & \leq  \sum_{i=1}^D\| u^{(i)} -  v^{(i)}\|^2_F,
\end{align*}
we obtain that 
\begin{align*}
\Eb\left[\left\|X_{u^{(1)},\ldots,u^{(D)}} - X_{v^{(1)},\ldots,v^{(D)}} \right\|_F^2\right] & = \sum_{i=1}^D\| u^{(i)} -  v^{(i)}\|^2_F \\
\end{align*}
Now let us consider another random process indexed by $(u^{(1)},\ldots,u^{(D)})$:
\begin{align*}
Y_{u^{(1)},\ldots,u^{(D)}} & \defeq \sum_{d=1}^D \langle g^{(d)}, u^{(d)}\rangle ,
\end{align*}
where $g^{(d)}\sim \Nc(0, I_{n_d})$, $d=1,\ldots,D$ 
are independent Gaussian random vectors. It is easy to see that the variance of the increments of $Y$ is
\begin{align*}
\Eb\left\| Y_{u^{(1)},\ldots,u^{(D)}} - Y_{v^{(1)},\ldots,v^{(D)}} \right\|^2
& = \Eb\left\| \sum_{d=1}^D\langle g^{(d)}, u^{(d)}- v^{(d)}\rangle  \right\|^2\\
& = \sum_{d=1}^D\left\| u^{(d)} -  v^{(d)}\right\|^2_F.
\end{align*}
Therefore, the variance of the increments of $Y$ is greater than or equal to the variance of the increments of $X$. 
Therefore, we can apply Slepian's lemma \cite{LedouxTalagrand:ProbabilityBanachSpaces13} and obtain 
\bean
\Eb \left[\|\Xc \|\right] &=&\Eb \left[\sup_{u^{(i)}:\|u^{(i)}\|=1\atop i=1,\ldots,D} X(u^{(1)},\ldots,u^{(D)})\right]\\
&\leq &  \Eb \left[\sup_{u^{(i)}:\|u^{(i)}\|=1\atop i=1,\ldots,D} Y(u^{(1)},\ldots,u^{(D)}) \right]\\
& = & \Eb \left[\sup_{u^{(i)}:\|u^{(i)}\|=1\atop i=1,\ldots,D} \sum_{i=1}^D\langle g^{(i)}, u^{(i)}\rangle \right] \\
& = & \Eb \left[\sum_{i=1}^D\|g^{(i)}\|\right] 
\eean
and by Jensen's inequality
\bean 
\Eb \left[\|\Xc \|\right] & \leq & \sum_{i=1}^D\sqrt{n_i}
\eean
where the last inequality is derived from Jensen's inequality. 
\end{proof}

\subsection{The entropy of the set of tensors with Frobenius norm equal to $\alpha$ and given Tucker rank}

Define the set 
\begin{align*}
\Tb_{r,\alpha} & = \left\{\Wc \in \mathbb R^{n_1\times\cdots\times n_D} \mid \Vert \Wc\Vert_F \le \alpha, \quad \sum_{d=1}^D {\rm rank}(\Wc^{(d)})=r\right\}. 
\end{align*}
Then, we have that 
\begin{theo}
For each $\epsilon$, the set $\Tb_{r,\alpha}$ has an $\epsilon'$-net of size $N(\epsilon)$ with 
\bean
N(\epsilon',r) &\leq & C_{D-1}^{r-1}  \left(\frac{3}{\alpha}\right)^{rn_*} \left(\frac{3\alpha}{\epsilon} \right)^{rn_*+ (r/D)^D}
\eean
where $n_*=\max\{n_1,\ldots,n_D\}$ and
\bean
\epsilon' = \epsilon  + \alpha\Big((1+ r\epsilon)^D -1\Big).
\eean
\end{theo}
\begin{proof}
For any $n$ and $\nu$ in $\mathbb N$ with $n\ge \nu$, let $\mathfrak{O}_{n,\nu}$ be the Stiefel manifold defined by 
\begin{align*}
\mathfrak{O}_{n,\nu} & = \left\{ U \in \Rb^{n\times \nu} \mid U^tU=I_\nu \right\}. 
\end{align*}
Then, it was proved in \cite[Lemma 3.1]{CandesPlan:IEEEIT11} that there exists an $\epsilon$-covering number of size 
$(9/\epsilon)^{n \nu}$. Let $\Pc(r,D)$ denote the set of integer partitions of $r$ using no more than 
$D$ integers. Now for each $\nu=(\nu_1,\ldots,\nu_D)\in \Pc(r,D)$, define the set 
\begin{align*}
\mathfrak L(\nu) & = \left\{ \Sc \in \mathbb R^{\nu_1 \times \cdots \times \nu_D} \mid 
\: \Vert \Sc\Vert_F=\alpha  \right\}.
\end{align*}
It is well known \cite{LedouxTalagrand:ProbabilityBanachSpaces13} that the unit sphere of $\mathbb R^m$ admits an $\epsilon$-net of size less than $(3/\epsilon)^m$. Using this fact, we easily obtain an $\epsilon-$net of $\mathfrak L(r)$ of size no larger than
$(3\alpha/\epsilon)^{\nu_1 \times \cdots \times \nu_D}$. 

Next, we are going to determine the size of $\epsilon$-net covering the set
\bean
\mathfrak{T}(\nu) = \left\{ \Sc\times_1 U^{(1)} \times_2  
\cdots \times_D U^{(D)}  \mid  \Sc\in\mathfrak{L}(\nu), U^{(d)}\in\mathfrak{O}_{n_d,\nu_d}, d=1,\ldots,D \right\}
\eean
Denote $\Delta \Sc = \Sc- \Sc_0$ and $\Delta U^{(d)} =  U^{(d)} - U_0^{(d)}$ for $d=1,\ldots,D$. Then
\bean
\Wc 
& = & \Sc\times_1 U^{(1)} \times_2  \cdots \times_D U^{(D)}  \\
& = & ( \Sc_0  + \Delta \Sc)\times_1 (  U^{(1)}_0 + \Delta U^{(1)}) \times_2  \cdots \times_D ( U^{(D)}_0 + \Delta U^{(D)}) \\
& = & \Sc_0\times_1 (  U^{(1)}_0 + \Delta U^{(1)}) \times_2  \cdots \times_D ( U^{(D)}_0 + \Delta U^{(D)}) \\
&& +\,\Delta\Sc\times_1 (  U^{(1)}_0 + \Delta U^{(1)}) \times_2  \cdots \times_D ( U^{(D)}_0 + \Delta U^{(D)}),
\eean
in which
\bean
&&\Sc_0\times_1 (  U^{(1)}_0 + \Delta U^{(1)}) \times_2  \cdots \times_D ( U^{(D)}_0 + \Delta U^{(D)}) \\
&=&\Sc_0\times_1 U^{(1)}_0 \times_2  \cdots \times_D U^{(D)}_0 
+  \sum_{d=1}^D\left(\Sc_0\times_d \Delta U^{(d)} \prod _{i=1\atop i\neq d}^D \times_i U^{(i)}_0 \right) + \cdots\\
&& +  \sum_{d_1, \ldots, d_k=1, 1\leq k\leq D\atop d_p\neq d_q,\forall p\neq q}^D
\left(\Sc_0\prod_{j=1}^k\times_{d_j} \Delta U^{(d_j)}\prod _{i=1\atop i\neq d_1,\ldots, d_k}^D \times_i U^{(i)}_0 \right) + \cdots\\
&& +\,\Sc_0\times_1 \Delta U^{(1)}_0 \times_2  \cdots \times_D \Delta U^{(D)}_0 .
\eean
Since
\bean
\Big\|\Delta\Sc\times_1 (  U^{(1)}_0 + \Delta U^{(1)}) \times_2  \cdots \times_D ( U^{(D)}_0 + \Delta U^{(D)})\Big\|_F &=& \|\Delta \Sc\|_F \\
\Big\|\Sc_0\prod_{j=1}^k\times_{d_j} \Delta U^{(d_j)}\prod _{i=1\atop i\neq d_1,\ldots, d_k}^D \times_i U^{(i)}_0 \Big\|_F & \leq  &\alpha \epsilon^k \prod_{j=1}^k \nu_{d_j}
\eean

It follows that
\bean
&&\|\Wc - \Sc_0\times_1 U^{(1)}_0 \times_2  \cdots \times_D U^{(D)}_0 \|_F \\
&\leq& \|\Delta\Sc\|_F + \sum_{d=1}^D\alpha \epsilon\nu_{d} + \cdots
+   \sum_{d_1, \ldots, d_k=1, 1\leq k\leq D\atop d_p\neq d_q,\forall p\neq q}^D\alpha \epsilon^k \prod_{j=1}^k \nu_{d_j}+ \cdots + \alpha \epsilon^D \prod_{d=1}^D \nu_{d}
\\
&\leq& \epsilon + \sum_{d=1}^D\alpha \epsilon r + \cdots
+   \sum_{d_1, \ldots, d_k=1, 1\leq k\leq D\atop d_p\neq d_q,\forall p\neq q}^D\alpha \epsilon^k \Big(\frac{r}{k}\Big)^k+ \cdots + \alpha \epsilon^D \Big( \frac{r}{D}\Big)^D \\
&=&\epsilon + \alpha\sum_{k=1}^D C_D^k  \epsilon^k \Big(\frac{r}{k}\Big)^k\\
&\leq & \epsilon  + \alpha\Big((1+ r\epsilon)^D -1\Big).
\eean

Now let us rewrite the set $\Tb$ as the following:
\bean
\Tb_{r,\alpha} & =& \left\{\Wc \in \mathbb R^{n_1\times\cdots\times n_D} \mid \Vert \Wc\Vert_F \le \alpha, \quad \sum_{d=1}^D {\rm rank}(\Wc^{(d)})=r\right\}\\
&=&\left\{ \Sc\times_1 U^{(1)} \times_2  \cdots \times_D U^{(D)} 
\mid \nu\in \Pc(r,D), \Sc\in \mathfrak{L}(\nu),\, U^{(d)}\in \mathfrak{O}_{n_d,\nu_d} ,d=1,\ldots,D
\right\}
\eean
and denote 
\bean
\epsilon' = \epsilon  + \alpha\Big((1+ r\epsilon)^D -1\Big).
\eean

Summing up our discussion, we conclude that there exists an $\epsilon$-net of $\Tb_{r,\alpha}$ with covering number
\begin{align*}
N(\epsilon',r) & \le  \sum_{\nu\in \Pc(r, D) }\: 
 \left(\frac{9}{\epsilon}\right)^{n_1\nu_1+\cdots + n_D\nu_D}
\left(\frac{3\alpha}{\epsilon}\right)^{\nu_1 \times \cdots \times \nu_D} ,
\end{align*}
and since the cardinality of $\Pc(r,D)$ equals $C_{D-1}^{r-1}$, we obtain that  
\bean
N(\epsilon',r) & \leq & C_{D-1}^{r-1} \max_{\nu\in\Pc(r,D)} 
  \left(\frac{9}{\epsilon}\right)^{n_1\nu_1+\cdots + n_D\nu_D}
\left(\frac{3\alpha}{\epsilon}\right)^{\nu_1 \times \cdots \times \nu_D}  \\
&\leq & C_{D-1}^{r-1} \max_{\nu\in\Pc(r,D)} 
  \left(\frac{9}{\epsilon}\right)^{n_1\nu_1+\cdots + n_D\nu_D}
\max_{\nu\in\Pc(r,D)}  
\left(\frac{3\alpha}{\epsilon}\right)^{\nu_1 \times \cdots \times \nu_D}
  \\
&\leq & C_{D-1}^{r-1}\left(\frac{9}{\epsilon}\right)^{rn_*}
 \left(\frac{3\alpha}{\epsilon} \right)^{ (r/D)^D} \\
& = &   C_{D-1}^{r-1}  \left(\frac{3}{\alpha}\right)^{rn_*} \left(\frac{3\alpha}{\epsilon} \right)^{rn_*+ (r/D)^D} ,
\eean
where $n_*=\max\{n_1,\ldots,n_D\}$.
\end{proof}

\subsection{The dual Romera-Paredes--Pontil function of Gaussian tensors}
In this section, we study the expected value of the evaluation at a Gaussian random tensor of the 
dual $\omega_\alpha^*$ of the Romera-Paredes--Pontil function $\omega_\alpha^{**}$. 
We will need some further notations. For any vector $s \in \mathbb R^{N}$. 
Using (\ref{dualRPP}), one easily obtains  
\begin{align}
\label{omstar}
\omega_\alpha^*(\Xc) & \leq  \sup_{\Vert w\Vert_2\le \alpha} \: \langle \sigma(\Xc),w\rangle-\Vert w\Vert_0.
\end{align}
We also have, by equation (7) in \cite{RomeraParedesPontil:NIPS13}, 
\begin{align*}
\omega_\alpha^*(\Xc) & \leq  \alpha \: \max_{r=0,\ldots,N} \: \Vert\sigma_{1:r}^\downarrow(\Xc)\Vert_2-r.
\end{align*}
We have the following result. 
\begin{lemm}\label{lemmaB}
Let 
\bean 
\Xc=\left(\Xc_{i_1,\ldots,i_D}\right)_{i_1=1,\ldots,n_1,\ldots,i_D=1,\ldots,n_D}.  
\eean 
be a tensor with i.i.d. standard Gaussian entries. Then
\bean
\label{dud0}
\Eb \left[\omega_\alpha^*(\Xc)\right] 
& \le &8 \sqrt{2} \: \alpha \:  
\Bigg(\sqrt{\log (p)} \frac{(1+N)^{D+1}-2^{D+1}}{D+1} + \frac{1}{4}\sqrt{\pi q} \Big(N+ \frac{N(1+N)^D-2^D}{D} \nonumber \\
& & +\frac{2^{D+1}-(1+N)^{D+1}}{D(D+1)}\Big) \Bigg) -\frac{N(1+N)}2\nonumber 
\eean
\end{lemm}
\begin{proof} 
Using (\ref{omstar}), and the tensor Von Neumann inequality \cite[Theorem 1]{ChretienWei:LAA15}, 
one obtains that
\begin{align}
\label{omstar}
\omega_\alpha^*(\Xc) & \le  \max_{r=1,\ldots, N} \: 
\sup_{\stackrel{\Vert \Wc\Vert_F\le \alpha}{\sum_{d=1}^D {\rm rank}(\Wc^{(d)})=r}} \: \langle \Xc,\Wc\rangle-r \\
& \le \alpha \: \sum_{r=1,\ldots,N} \: 
\sup_{\stackrel{\Vert \Wc\Vert_F\le 1}{\sum_{d=1}^D {\rm rank}(\Wc^{(d)})=r}} \: \langle \Xc,\Wc\rangle-r 
\end{align}
Since we must enforce the constraint $\Vert \Wc\Vert_F\le 1$, Dudley's entropy bound says that 
\cite{LedouxTalagrand:ProbabilityBanachSpaces13}
\begin{align}
\label{dud}
\Eb \left[\omega_\alpha^*(\Xc)\right] & = 8 \sqrt{2} \: \alpha \: \sum_{r=1,\ldots,N} \:  
\int_{0}^{1} \sqrt{\log(N(\epsilon))} \: d\epsilon-r.  
\end{align}
Let us compute the integral term. We have
\bean
&&\int_{0}^{1} \sqrt{\log(N(\epsilon'))} \: d\epsilon' \\
&=& \int_{0}^{1} \left[\log\left(C_{D-1}^{r-1}  \left(3\right)^{rn_*} \left(\frac{3}{\epsilon} \right)^{rn_*+ (r/D)^D}
\right)\right]^{1/2} \Big(1+r D (1+r\epsilon)^{D-1}\Big) d\epsilon \\
&\leq & \int_0^{1} \sqrt{\log (p)}(1+ r D (1+r\epsilon)^{D-1})d\epsilon 
+\int_0^{1} \sqrt{q \log (3/\epsilon)}(1+ r D (1+r\epsilon)^{D-1})d\epsilon,
\eean
where $p=C_{D-1}^{r-1} 3^{rn_*}$ and $q=rn_*+ (r/D)^D$.
We have
\bean
&&\int_0^{1} \sqrt{\log (p)}(1+ r D (1+r\epsilon)^{D-1})d\epsilon \\
&=& \sqrt{\log(p)} + r D \sqrt{\log (p)}  \int_0^{1} (1+r\epsilon)^{D-1}d\epsilon \\
&=& \sqrt{\log(p)} + r D \sqrt{\log (p)}  \int_1^{r+1} z^{D-1}
\frac{1}{r}
dz \\
&=& \sqrt{\log(p)} + \sqrt{\log (p)}  ( (r +1)^D -1 ) \\
&=& \sqrt{\log (p)}  (r +1)^D,
\eean
and
\bean
& &\int_0^{1} \sqrt{q \log (3/\epsilon)}(1+ r D (1+r\epsilon)^{D-1})d\epsilon \\
&\leq &\sqrt{q}\Big(1+ r D (1+r)^{D-1}\Big)\int_0^{1} \sqrt{\log (3/\epsilon)}d\epsilon \\
&=&\sqrt{q}\Big(1+ r D (1+r)^{D-1}\Big)\int_{\sqrt{\log 3}}^{+\infty} x^2 e ^{-x^2}dx \\
&<&\frac{1}{4}\sqrt{\pi q}\Big(1+ r D (1+r)^{D-1}\Big).
\eean
Therefore
\bean
&&\int_{0}^{1} \sqrt{\log(N(\epsilon'))} \: d\epsilon' \\
& \leq & \sqrt{\log (p)}  (1+r)^D + \frac{1}{4}\sqrt{\pi q}\Big(1+ r D (1+r)^{D-1}\Big),
\eean
where $p=C_{D-1}^{r-1} (3)^{rn_*}$ and $q=rn_*+ (r/D)^D$. Plugging this result into (\ref{dud}), we obtain 
\begin{align}
\label{dud2}
\Eb \left[\omega_\alpha^*(\Xc)\right] & \le 8 \sqrt{2} \: \alpha \: \sum_{r=1,\ldots,N} \:  
\Bigg(\sqrt{\log (p)}  (1+r)^D + \frac{1}{4}\sqrt{\pi q}\Big(1+  r D (1+r)^{D-1}\Big) \Bigg) -r.  
\end{align}
Approximating sums by integrals, we thus obtain 
\begin{align}
\label{dud2}
\Eb \left[\omega_\alpha^*(\Xc)\right] & \le 8 \sqrt{2} \: \alpha \:  
\Bigg(\sqrt{\log (p)} \frac{(1+N)^{D+1}-2^{D+1}}{D+1}  + \frac{1}{4}\sqrt{\pi q} \Big(N+ \frac{N(1+N)^D-2^D}{D} \nonumber \\
& +\frac{2^{D+1}-(1+N)^{D+1}}{D(D+1)}\Big) \Bigg) -\frac{N(1+N)}2\nonumber 
\end{align}
as announceed. 
\end{proof}


\bibliographystyle{amsplain}
\bibliography{Biblio}{}

\end{document}